\documentclass[a4paper]{amsart}
\usepackage[latin9]{inputenc}
\usepackage{a4}
\usepackage[all]{xy}
\usepackage{amsfonts}
\usepackage{amssymb}
\usepackage{amsmath}
\usepackage{amsthm}
\usepackage{changepage}
\usepackage{nomencl}
\usepackage{geometry}
\usepackage{tikz}
\usepackage{verbatim}
\usetikzlibrary{matrix}
\usepackage[toc,page]{appendix}
\begin{document}
\providecommand{\keywords}[1]{\textbf{\textit{Keywords: }} #1}
\newtheorem{thm}{Theorem}[section]
\newtheorem{lemma}[thm]{Lemma}
\newtheorem{prop}[thm]{Proposition}
\newtheorem{cor}[thm]{Corollary}
\theoremstyle{definition}
\newtheorem{defi}[thm]{Definition}
\theoremstyle{remark}
\newtheorem{remark}[thm]{Remark}
\newtheorem{prob}[thm]{Problem}
\newtheorem{conjecture}[thm]{Conjecture}
\newtheorem{ques}[thm]{Question}

\newcommand{\cc}{{\mathbb{C}}}   
\newcommand{\ff}{{\mathbb{F}}}  
\newcommand{\nn}{{\mathbb{N}}}   
\newcommand{\qq}{{\mathbb{Q}}}  
\newcommand{\rr}{{\mathbb{R}}}   
\newcommand{\zz}{{\mathbb{Z}}}  
\newcommand{\K}{\mathbb{K}}

\title{On unramified solvable extensions of small number fields}
\author{Joachim K\"onig}
\address{Department of Mathematics Education, Korea National University of Education, Cheongju, South Korea}
\email{jkoenig@knue.ac.kr}
\begin{abstract}
We investigate unramified extensions of number fields with prescribed solvable Galois group and certain extra conditions. In particular, we are interested in the minimal degree $d'$ of a number field $K$, Galois over $\mathbb{Q}$, such that $K$ possesses an unramified $G$-extension. We improve the best known bounds for the degree of such number fields $K$ for certain classes of solvable groups, in particular nilpotent groups. 
%We also provide conjectures, related to the Malle-conjecture and the Cohen-Lenstra heuristics, on the number of $G$-extensions of $K$ of bounded discriminant which fulfill the above conditions.
\end{abstract}
\maketitle

%(Convention: In this paper, the notion ``unramified extension" of number fields concerns only the non-archimedean primes, unless explicitly stated otherwise.)
%\marginpar{So far, not true everywhere (only specialization section)!}
\section{Introduction}
A problem of widespread interest in algebraic number theory is the construction of unramified extensions $L/K$ of number fields with prescribed Galois group. It is well-known (e.g., as a direct consequence of results on $S_n$-extensions with squarefree discriminant, cf.\ \cite{Uchida}, \cite{Yamamoto}) that such extensions exist for any given group $G$. A more interesting question is, what is the smallest degree of such a number field $K$ over $\mathbb{Q}$, possibly with additional requirements on $K$.
In the following, we denote by $d(\mathbb{Q},G)$ the smallest integer $d$ such that there exists a number field $K$ of degree $d$ over $\mathbb{Q}$, such that $K$ possesses an unramified Galois extension with group $G$; and by $d'(\mathbb{Q},G)$ the smallest integer $d'$ as above such that $K/\mathbb{Q}$ is additionally Galois. It is commonly conjectured that every finite group occurs as the Galois group of an unramified Galois extension $L/K$, where $K$ is a quadratic number field, i.e., $d(\mathbb{Q},G)=d'(\mathbb{Q},G)=2$. However, this is a difficult question (in class field theory) even for the case of abelian groups. Detailed heuristics predicting the distribution of such extensions, generalizing the Cohen-Lenstra heuristics, have been developed by Wood (\cite{Wood}).

For solvable $G$, it is known from work of Kim building on Shafarevich's method (\cite{KKS2}, \cite{KKS3}) that $d'(\mathbb{Q},G) \le \exp(G)$, where the {\it exponent} $\exp(G)$ of $G$ is defined as the least common multiple of all element orders in $G$. Previously, Nomura (\cite{Nomura}) had given the bound $d'(\mathbb{Q},G)\le p\cdot |\Phi(G)|$ for $p$-groups $G$, with $\Phi(G)$ the Frattini subgroup of $G$. As noted in \cite[Remark 5.2]{KKS2}, this bound is always $\ge \exp(G)$. %This in term improved on previous bounds by Nomura.

We additionally define $e(\mathbb{Q},G)$ as the minimal number $e$ such that $\mathbb{Q}$ admits a tamely ramified $G$-extension all of whose ramification indices divide $e$. The relevance of this definition for the original question on unramified $G$-extensions is due to Abhyankar's lemma, which shows immediately that $d(\mathbb{Q},G) \le e(\mathbb{Q},G)$ (cf.\ \cite[Lemma 2.1]{KNS}). 

%Let $G$ be a finite group. 
For a finite group $G$, define the {\it generator exponent} of $G$ to be  $$\textrm{ge}(G):=\min_S \textrm{lcm}\{ord(x)\mid x\in S\},$$ where $S$ ranges over all generating subsets of $G$. 
It is easy to see that $e(\qq,G)\ge \textrm{ge}(G)$ for all finite groups $G$. This is because the set of all inertia groups of a (tamely ramified) $G$-extension has to generate $G$.
The converse is open.

\begin{ques}
\label{conj:e3}
Let $G$ be a finite group. 
\iffalse and let $S$ be a set of generators of $G$. 
%Then there exist infinitely many pairwise linearly disjoint, 
Is there a tamely ramified extension $L/\qq$ with Galois group $G$ such that all inertia groups at ramified primes in $L/\qq$ are generated by a conjugate of an element of $S$?\\
In other words, d\fi
Does $e(\qq,G)$ equal $\textrm{ge}(G)$?
\end{ques}

Note that while bounds on $e(\mathbb{Q},G)$ do not automatically yield bounds on $d'(\mathbb{Q},G)$ in general, they do as soon as the implied tamely ramified $G$-extensions satisfy certain additional local conditions, cf.\ Lemma \ref{lem:unram}. We therefore connect Question \ref{conj:e3} with the following, which is more accessible than the stronger conjecture $d'(\mathbb{Q},G)=2$.
\begin{ques}
\label{conj:d}
Let $G$ be a finite group. Is it true that $d'(\qq,G)\le \textrm{ge}(G)$?
\end{ques}

In the previous paper \cite{KNS}, Question \ref{conj:e3} was investigated using function field methods, with a focus on nonsolvable groups, in particular reaching the best possible bound $d(\mathbb{Q},G) = e(\qq,G)=\textrm{ge}(G) =2$ for several new groups.

Here, we instead focus on solvable groups. 
For certain classes of groups, in particular for so-called regular $p$-groups, it holds that $\textrm{ge}(G)=\exp(G)$, meaning that already the aforementioned results \cite{KKS2}, \cite{KKS3} yield a positive answer to Questions \ref{conj:e3} and \ref{conj:d} for such groups. In particular, since it is known that  all $p$-groups of nilpotency class $\le p-1$ are regular, it follows that $d'(\qq,G) \le e(\qq, G)=\textrm{ge}(G)$ for those groups.

The main goal of this note is to extend the above beyond the special case $\exp(G)=\textrm{ge}(G)$. In particular, we prove:
\begin{thm}
\label{thm:main}
Let $G$ be a nilpotent group of nilpotency class $\le p$, where $p$ is the smallest prime divisor of $|G|$. Then $d'(\mathbb{Q},G)\le e(\mathbb{Q},G) = \textrm{ge}(G)$. More precisely, there exist infinitely many cyclic number fields $K$ of degree $\le\textrm{ge}(G)$ such that $K$ possesses an unramified $G$-extension. %In particular, $d'(\mathbb{Q},G) \le \textrm{ge}(G)$.
\end{thm}

Note that groups of nilpotency class $p$ include many groups for which $\textrm{ge}(G)$ is strictly smaller than $\exp(G)$, making Theorem \ref{thm:main} an improvement over previously available bounds. An easy example (but far from the only one) is the wreath product $G=C_p\wr C_p (=(C_p)^p \rtimes C_p)$, whose nilpotency class is $p$,\footnote{Indeed, the nilpotency class of a $p$-group of order $p^k$ ($k\ge 2$) is always bounded from above by $k-1$, and in case of nilpotency class $<p$, the following discrepancy between exponent and generator exponent would be impossible, see Section \ref{sec:group}.} generator exponent is $p$ and exponent is $p^2$.

Before proving Theorem \ref{thm:main} in Section \ref{sec:proof}, we discuss some methods allowing generalizations in Section \ref{sec:compat}, in particular providing positive answers to Questions \ref{conj:e3} and \ref{conj:d} for certain classes of $p$-groups of arbitrarily high nilpotency class.

\section{Some prerequisites}
\label{sec:group}
We recall some standard notions and elementary results, mostly from group theory, which will be used later. The first is the notion of a regular $p$-group. One of several equivalent definitions is the following (see \cite[Chapter III.10]{Huppert}).

\begin{defi}[Regular $p$-group]
A $p$-group $G$ is called regular if for every $a,b\in G$ there exists $c$ in the derived subgroup of $\langle a,b\rangle \le G$ such that $a^pb^p = (ab)^p c^p$.
\end{defi}

We will only use the following two consequences of regularity, cf.\ Corollary 4.13 and Theorem 4.26 in \cite{Hall}.
\begin{prop}
Every $p$-group of nilpotency class $<p$ is regular.
\end{prop}

\begin{prop}
\label{prop:expge}
In a regular $p$-group $G$, the order of a product of any finitely many elements cannot exceed the orders of all these elements. In particular, $\exp(G) = \textrm{ge}(G)$.
\end{prop}

We will also make use of higher commutators and their role in calculating powers of products of group elements.
\begin{defi}
Let $G$ be a finite group and $a,b\in G$. Denote by $[a,b]:=a^{-1}b^{-1}ab$ the commutator of $a$ and $b$. Iteratively, a {\it commutator of weight $i$} in $a$ and $b$ is defined as follows:
\begin{itemize}
\item The commutators of weight $1$ are $a$ and $b$.
\item The commutators of weight $i\ge 2$ are $[x,y]$ where $x$ and $y$ are commutators of weight $j$ and $i-j$ for some $j\in \{1,...,i-1\}$.
\end{itemize}
\end{defi}

The following is Theorem 3.1 in \cite{Hall}.
\begin{thm}[P.\ Hall]
\label{thm:hall}
Let $G$ be a finite group and $a,b\in G$. 
For $i\in \nn$ denote by $R_{i,j}$ the iterated commutators of weight $i$ in $a$ and $b$ ($j\in \{1,...,n_i\}$ for some $n_i\in \nn$), in some prescribed order.
Then there exist polynomial functions $f_{i,j}$ such that for all $n\in \nn$:
$$(ab)^n = a^n b^n \prod_{i\ge 2}\prod_{j=1}^{n_i} R_{i,j}^{f_{i,j}(n)}.\footnote{The product over $i\ge 2$ is a priori infinite, and should be interpreted as ``$\prod_{2\le i <N} (...)$ times an element of the group generated by weight $N$ commutators", for arbitrarily chosen $N\in \mathbb{N}$. In groups where all suitably high commutators vanish (such as nilpotent groups), there is no ambiguity in the notation.}$$
More precisely, $f_{i,j}$ is an integer linear combination of the polynomials $\binom{X}{1},...,\binom{X}{i}$, where $\binom{X}{d}:=\frac{X(X-1)\cdots (X-d+1)}{d!}$.
\end{thm}

We set $G_1:=G$ and iteratively $G_d:=[G,G_{d-1}]$ for every $d\ge 2$. In particular $G_2 = [G,G] = G'$ is the commutator subgroup of $G$, and $G=G_1>G_2>\dots$ is the {\it lower central series} of $G$. In particular, if $G$ is nilpotent of class $c$, then $G_{c+1}=\{1\}$. Following \cite[Theorems 2.51 and 2.53]{Hall}, one has:
\begin{lemma}
\label{lem:highcomm}
For all $i,j\ge 1$, one has $[G_i, G_j] \le G_{i+j}$. In particular, every weight $i$ commutator of $G$ is contained in $G_i$.
\end{lemma}

Finally, we include a number theoretic lemma which ensures $d'(\mathbb{Q},G)\le e(\mathbb{Q},G)$ under certain extra conditions.
\begin{lemma}
\label{lem:unram}
Let $G$ be the Galois group of a tamely ramified extension $F/\mathbb{Q}$ all of whose decomposition groups are abelian. Then $G$ occurs as the Galois group of an unramified extension of some cyclic number field $L$. Moreover, let $m$ denote the least common multiple of all ramification indices at ramified primes in $F/\mathbb{Q}$. Then one may choose $L$ such that $[L:\mathbb{Q}]\le m$. 
 \end{lemma}
 \begin{proof}
 This is Lemma 4.5 in \cite{KK}.
 \end{proof}

\section{Shafarevich's method and the constant $r(G)$}
 The following deep result, due to Shafarevich (see, e.g., \cite[Chapter IX.6]{NSW}), solves the inverse Galois problem for solvable groups.
\begin{thm}[Shafarevich]
\label{thm:shaf}
Let $G$ be a finite solvable group and $K$ a number field. % and $S$ a finite set of primes of $K$. 
Then there are infinitely many Galois extensions $L/K$ with group $G$ fulfilling the following:
\begin{itemize}
\item[i)] $L/K$ is tamely ramified.%, and unramified at $S$.
\item[ii)] All decomposition groups at ramified primes in $L/K$ are cyclic and equal the respective inertia groups.
\end{itemize}
\end{thm}

Since decomposition groups at unramified primes are automatically cyclic, Theorem \ref{thm:shaf} together with Lemma \ref{lem:unram} regain immediately the bound $d'(\mathbb{Q},G)\le \exp(G)$ for all solvable groups $G$. In order to improve on this bound and move towards the proof of Theorem \ref{thm:main}, we recall Shafarevich's method in more detail.

Firstly, at the heart of Shafarevich's proof of Theorem \ref{thm:shaf} is a result on solvability of split embedding problems with nilpotent kernel (see \cite[Theorem 9.6.7]{NSW}), which, given a Galois extension $L/K$ with group $H$, guarantees the existence of a $N\rtimes H$-extension $F/K$ containing $L/K$ such that all ramified primes of $L/K$ split completely in $F/L$ and all ramified primes of $F/L$ have cyclic decomposition groups equal to the respective inertia group in $F/K$. 

Next, given any solvable group $G$ and normal subgroup $N\triangleleft G$, call a proper subgroup $U< G$ a {\it partial complement} for $N$ if $NU=G$. Note that in this case $G$ necessarily occurs as a quotient of a suitable semidirect product $N\rtimes U$. Partial complements exist for all normal subgroups $N$ not contained in the Frattini subgroup of $G$ (\cite[Prop.\ 9.6.8]{NSW}). In particular, the Fitting subgroup $F(G)$, defined as the (unique) largest nilpotent normal subgroup of $G$, always has this property (\cite[Prop.\ 9.6.9]{NSW}). Since $|U|<|G|$, Theorem \ref{thm:shaf} is then derived by induction, using that Properties i) and ii) are preserved under taking quotients.

The above motivates the following definition.
\begin{defi}
\label{def:scholzexp}
Let $G$ be a solvable group. 
Set $G_0:=G$. As long as $G_{i-1}\ne \{1\}$, we iteratively define $N_i$ to be a nilpotent normal subgroup of $G_{i-1}$ such that $N_i$ possesses a partial complement $G_i$ in $G_{i-1}$ (i.e. $G_i\ne G_{i-1}$ and $N_iG_i=G_{i-1}$). Let $s$ be minimal such that $G_s=\{1\}$.
%Note that this series will terminate at $G_r=\{1\}$ for some $r\in \nn$, simply because $|G_{i}|<|G_{i-1}|$. 
%Note also that every solvable subgroup $G$ possesses at least one nilpotent normal subgroup with a partial complement, namely the Fitting subgroup $F(G)$.
For each $i=1,...,s$ denote by $e_i$ the exponent of the group $N_i$.
Define $r(G)$ as $\min \textrm{lcm}(e_1,...,e_{s})$, where the minimum is taken over all series of $(N_i,G_i)_{i=1,...,s}$ as above.
\end{defi}

Note in particular that $r(G)$ divides $\exp(G)$, as it is the least common multiple of certain element orders of $G$.
For many groups $G$, $r(G)$ is actually significantly smaller than $\exp(G)$. For example, let $G = C_p\wr (C_p \wr(...\wr C_p))...)$ be a $k$-fold iterated wreath product of cyclic groups of order $p$. Then $\exp(G)=p^k$, whereas $r(G)=p$. To see the latter, simply write $G=(C_p)^n\rtimes H$ with suitable $n\in \mathbb{N}$, set $N_1:=(C_p)^n$, $G_1:=H$ and note that $\exp(N_1)=p$ and $G_1$ is essentially of the same structure as $G$, so one can proceed by induction.
On the other hand, one always has $r(G)\ge \textrm{ge}(G)$, since  $N_1\cdots N_r = G$.

The following useful inequality is also straightforward from the definition of $r(G)$.
\begin{lemma}
\label{lem:exp_char}
Let $G$ be a $p$-group, $N$ a normal subgroup of $G$ and $U$ a partial complement of $N$ in $G$.
%$((N_i,G_i)\mid i\in \{1,...,s\})$ a series as in Def.\ \ref{def:scholzexp} which reaches the minimal value $\textrm{lcm}(e_1,...,e_{s})=r(G)$. Then $r(G) = \max\{\exp(N_1), r(G_1)\}$.
Then $r(G)\le \textrm{lcm}\{\exp(N), r(U)\}$.
\end{lemma}
\begin{proof}
Set $N_1=N$, $G_1 = U$, and continue $(N_1,G_1)$ to a series $((N_i,G_i) \mid i\in \{1,\dots, s\})$ as in Def.\ \ref{def:scholzexp}, and such that the series $((N_i,G_i) \mid i\in \{2,\dots, s\})$ inside $U$ reaches the smallest possible value $r(U)$. Set $e_i = \exp(N_i)$ for $i=1,\dots, s$.
We have $r(U)= \textrm{lcm}(e_2,\dots, e_s)$, and $r(G)\le \textrm{lcm}(e_1,\dots, e_s) = \textrm{lcm}(\exp(N), r(U))$. \end{proof}

\begin{prop}\label{solv_cyclic_decomp}
Let $G$ be a solvable group and $k$ be a number field. Let $((N_i,G_i)\mid i\in \{1,...,s\})$ be any series of nilpotent normal subgroups $N_i$ and partial complements $G_i$ as in Def.\ \ref{def:scholzexp}, and let $e_i = \exp(N_i)$ for $i=1,\dots, s$.
Then there exist infinitely many tamely ramified Galois extensions $F/k$ with group $G$ such that all ramification indices divide $\textrm{lcm}(e_1,\dots, e_s)$, and all decomposition groups at ramified primes are cyclic, equal to the inertia groups.
Moreover, there exist infinitely cyclic Galois extensions $K/k$ of degree $[K:k]\le \textrm{lcm}(e_1,\dots, e_s)$ such that $K$ possesses an unramified $G$-extension. 
In particular, $d'(\mathbb{Q},G)$ and $e(\mathbb{Q},G)$ are bounded from above by $r(G)$.
\end{prop}
\begin{proof}
It suffices to prove the first assertion, since the second one follows from Lemma \ref{lem:unram}, and the last one is immediate from the definition of $r(G)$.
Proceed by induction over $s$.\\
If $s=1$, then $G$ is nilpotent of exponent $e_1$, and the assertion is immediate from Theorem \ref{thm:shaf}.% and Lemma \ref{lem:unram}.
Now, let $s\ge 2$. Then $G=N_1G_1$ is a quotient of some semidirect product $N_1\rtimes G_1$. 
Note that $((N_i,G_i)\mid i\in \{2,...,s\})$ is a series as in Def.\ \ref{def:scholzexp} for the group $G_1$.
Thus, we may inductively assume the existence of a $G_1$-extension $F/k$ yielding the assertion for $G_1$.
By \cite[Theorem 9.6.7]{NSW}, there exist infinitely many tamely ramified Galois extensions $E/k$ with group $N_1\rtimes G_1$ such that all decomposition groups at ramified primes are cyclic, equal to the respective inertia groups, and embed either into $Gal(F/k)$ or into $N_1$.
Thus, all ramification indices in $E/k$, and a fortiori in its $G$-subextension, divide $\textrm{lcm}(e_1, \textrm{lcm}(e_2,\dots, e_s))$. This completes the proof.%
\end{proof}

\section{Groups satisfying $r(G) = ge(G)$}
\subsection{Compatibility with taking direct products and wreath products}
\label{sec:compat}
Proposition \ref{solv_cyclic_decomp} shows that the Shafarevich method yields the constant $r(G)$, rather than the in general larger $\exp(G)$, as an upper bound for $e(\mathbb{Q},G)$ and $d'(\mathbb{Q},G)$. However, the true value of $r(G)$ is usually hard to determine directly from its definition. We therefore aim at exhibiting examples in which $r(G)=\textrm{ge}(G)$, thus providing a positive answer to Questions \ref{conj:e3} and \ref{conj:d} for $G$. 
We begin with a simple, but useful observation.
\begin{lemma}
\label{lem:dirprod}
Let $G = G_1\times \cdots \times G_n$ be solvable, and assume $\textrm{ge}(G_i) = r(G_i)$ for all $i=1,\dots, n$. Then $\textrm{ge}(G) = r(G)$.
\end{lemma}
\begin{proof} Since each generating set of $G$ projects to a generating set of each $G_i$, and conversely the union of generating sets for each $G_i$ forms a generating set for $G$, one has $\textrm{ge}(G) = \textrm{lcm}(\textrm{ge}(G_1),\dots, \textrm{ge}(G_n))$. 

Regarding $r(G)$, let $((N_{ij},G_{ij})\mid i\in \{1,\dots ,s\})$ be a series of normal subgroups and partial complements inside $G_j$ as in Definition \ref{def:scholzexp} (assumed of the same length $s$ independent of $j$, via adding trivial subgroups if necessary), reaching the minimum value $r(G_j)$ for $j=1,\dots, n$, then $((\prod_{j=1}^n N_{ij}, \prod_{j=1}^n G_{ij}), \mid i\in \{1,\dots, s\})$ reaches the value $\textrm{lcm}(\exp(\prod_{j=1}^n N_{1j}), \dots, \exp(\prod_{j=1}^n N_{sj})) = \textrm{lcm} \{\exp(N_{i,j})\mid i=1,\dots, s; j=1,\dots, n\} = \textrm{lcm}(r(G_1),\dots, r(G_n))$. In particular, this shows $r(G) \le \textrm{lcm}(r(G_1),\dots, r(G_n)) = \textrm{lcm}(\textrm{ge}(G_1),\dots, \textrm{ge}(G_n)) = \textrm{ge}(G)$. Since always $\textrm{ge}(G) \le r(G)$, the assertion follows.
\end{proof}

In other words, the equality $r(G) = \textrm{ge}(G)$ is well-behaved under taking direct products. It is also well-behaved under taking wreath products, at least under some technical assumptions.

\begin{lemma}
\label{lem:wreath}
Let $G$ and $H$ be solvable groups, with $H$ embedded into $S_n$, and let $\Gamma = G\wr H = G^n \rtimes H$, with $H$ acting by permuting the $n$ copies of $G$. If $\textrm{ge}(G) = r(G)$ and $\textrm{ge}(H) = r(H)$, then $\textrm{ge}(\Gamma) = r(\Gamma)$, provided that at least one of the following is fulfilled.
\begin{itemize}
\item[a)] $\textrm{ge}(G)$ divides $\textrm{ge}(H)$.
\item[b)] $G$ has a cyclic quotient $C$ of order $\textrm{ge}(G)$.\footnote{This is automatic e.g.\ if $G$ is abelian; but also, e.g., if $\textrm{ge}(G)$ is a prime.}
\end{itemize}
\end{lemma}
\begin{proof}
Let $((N_i,G_i)\mid i\in \{1,...,s\})$ as in Def.\ \ref{def:scholzexp}, achieving the minimal value $\textrm{lcm}(e_1,...,e_{s})=r(G)$. Set $\tilde{N_i} = N_i^n \le G^n$ and $\tilde{G_i} = G_i\wr H$. Then $\tilde{G_i}$ is a partial complement for the normal subgroup $\tilde{N_i}$ of $\tilde{G_{i-1}}$, and $\tilde{G_s} = H$. Continue this sequence by choosing a sequence of normal subgroups and partial complements inside $H$ achieving the minimal value $r(H)$. One obtains $r(\Gamma) \le \textrm{lcm}(\exp(\tilde{N_1}),\dots, \exp(\tilde{N_s}), r(H))$. Noting that $\exp(\tilde{N_i})=\exp(N_i)$ for all $i$, one obtains $r(\Gamma) \le \textrm{lcm}(r(G), r(H)) = \textrm{lcm}(\textrm{ge}(G), \textrm{ge}(H))$.

Now in case a), the latter expression simply equals $\textrm{ge}(H)$, which is a trivial lower bound for $\textrm{ge}(\Gamma)$, via projecting a generating set onto one of $H$. In total $r(\Gamma) \le \textrm{ge}(\Gamma)$, and thus equality, as claimed. In case b), $\Gamma$ projects onto $C\wr H$, which (due to $C$ being abelian) projects onto $C\times H$. Thus $\textrm{ge}(\Gamma) \ge \textrm{ge}(C\times H) = \textrm{lcm}(\textrm{ge}(C), \textrm{ge}(H)) = \textrm{lcm}(\textrm{ge}(G), \textrm{ge}(H))$, with equality $r(\Gamma) = \textrm{ge}(\Gamma)$ in total.
\end{proof}

\begin{remark}
Lemmas \ref{lem:dirprod} and \ref{lem:wreath} yield a mechanism to construct large classes of groups with a positive answer to Questions \ref{conj:e3} and \ref{conj:d}, by beginning, e.g., with groups as in Theorem \ref{thm:main} and taking iterated direct and wreath products.
For example, taking iterated wreath products of a $p$-group $G$ of nilpotency class $\le p$ yields examples $\Gamma$ of arbitrarily high nilpotency class, whereas starting with, e.g., a nilpotent group $G$ of non prime-power order necessarily yields non-nilpotent examples $\Gamma$ (see, e.g., \cite{Baum}). 
 It should be remarked that the stronger condition $\exp(G) = \textrm{ge}(G)$, while also preserved under taking direct products, is not at all preserved under taking wreath products. In fact, when taking iterated wreath products of a group $G$ with itself, the generator exponent is preserved, whereas the exponent grows in every iteration. This serves as additional motivation for investigation of the constant $r(G)$, since it allows automatic construction of classes of examples which would be missed by naive considerations investigating only $\exp(G)$.
\end{remark}

\subsection{Proof of Theorem \ref{thm:main}}
\label{sec:proof}
We now turn to the proof of Theorem \ref{thm:main}. This involves a close inspection of commutators in nilpotent groups.

\begin{lemma}
\label{lem:main}
Let $G$ be a $p$-group of generator exponent $e:=\textrm{ge}(G)$ and nilpotency class $c$. If 
$p\ge c$, then $G'=[G,G]$ is of exponent at most $e$.
\end{lemma}
\begin{proof}
We show iteratively that $G_d$ is of exponent at most $e$ for $d=c+1,...,2$ in inverse order. 
The statement is trivial for $G_{c+1}=\{1\}$.\\
So assume the statement has been shown for $G_{d+1}$ (for some $d\ge 2$).
Since $G_d$ is of nilpotency class $\le c-1 < p$, it is regular. Thus by Proposition \ref{prop:expge}, it suffices to show $\textrm{ge}(G_d) \le e$. I.e., it suffices to show that every {\it commutator} in $G_d = [G,G_{d-1}]$ has order dividing $e$.

Let $\{x_1,\dots, x_n\}$ be a generating set of $G$ with all $x_i$ of order dividing $e$ (which exists by assumption).
Using the well-known commutator identity \begin{equation}
\label{eq:1}
[xz,y]=[z,[y,x]][x,y][z,y]\end{equation} iteratively, every commutator $[g,h]$ (with $g\in G$, $h\in G_{d-1}$) can be written as a product of commutators of the form $[x_{n_i},h_i]$ with $n_i\in \{1,...,n\}$ and $h_i\in G_{d-1}$. In particular, $[g,h]^e = (\prod_i [x_{n_i},h_i])^e$.
Again since $G_d$ is regular, the order of $[g,h]$ cannot exceed all the orders of $[x_{n_i},h_i]$. It thus suffices to show $[x_k,y]^e=1$ for $k\in \{1,...,n\}$ and $y\in G_{d-1}$. 
We have
$$1=[1,y]=[x_k^e,y]=[x_k,y]^{x_k^{e-1}} \cdot [x_k^{e-1},y] = [x_k,y]^{x_k^{e-1}} \cdots [x_k,y]^{x_k}\cdot[x_k,y].$$
Using the fact that $x_k^{e-1}=x_k^{-1}$, the above equation simplifies to
$$1=[x_k^e,y]=(x_k\cdot [x_k,y])^e.$$
Write the last power $(x_k\cdot [x_k,y])^e$ out using Theorem \ref{thm:hall}. We obtain $1=\underbrace{x_k^e}_{=1}[x_k,y]^e$ times terms of the form $R_{i,j}^{f_{i,j}(e)}$, with weight $i\ge 2$ commutators $R_{i,j}$ of $x_k$ and $[x_k,y]$ and polynomials $f_{i,j}$ which are integer linear combinations of $\binom{X}{1}$,..., $\binom{X}{i}$. In particular, all these higher commutators lie in $[G,[G,G_{d-1}]] = G_{d+1}$. Therefore they all have order dividing $e$, by induction. 
Furthermore, using the fact that $[x_k,y]\in G_d$ and the fact that all higher commutators $R_{i,j}$ as above may be assumed to contain at least one entry $[x_k,y]$, Lemma \ref{lem:highcomm} yields $R_{i,j}\in [G_d, G_{i-1}] \le G_{d+i-1}$. In particular, $R_{i,j}$ vanishes for all $i\ge c-d+2$. So we may assume $i\le c-d+1 \le c-1 < p$. But then $i!$ is coprime to $p$ and hence $f_{i,j}(e)$ is divisible by $e$, implying $R_{i,j}^{f_{i,j}(e)}=1$.

%the higher commutators $R_{i,j}$ occur only for weight $i\le c-d+1 (\le c-1)$ (because at least one entry in this commutator must be $[x_k,y]$). Therefore the assumption $c\le p$ is sufficient to ensure that all $f_{i,j}(e)$ are divisible by $e$. We then have $R_{i,j}^{f_{i,j}(e)}=1$ for all the higher commutators $R_{i,j}$. 
Therefore finally $[x_k,y]^e=1$. This shows the assertion.
\end{proof}

\begin{thm}
\label{thm:pgroups}
For any $p$-group $G$ of nilpotency class $c\le p$, it holds that $r(G) = \textrm{ge}(G)$.
%In particular, there exist infinitely many Galois extensions $K|\qq$ of degree $[K:\qq]$ at most $e$ such that $G$ occurs as the Galois group of an unramified Galois extension of $K$.
\end{thm}
\begin{proof}
Let $\{x_1,...,x_k\}$ be a minimal set of generators such that all $x_i$ have order dividing $e:=\textrm{ge}(G)$. Assume $k\ge 2$ without loss.
Set $H:=\langle G', x_k\rangle$, and consider the commutator subgroup $H'$. Using the commutator identitiy \eqref{eq:1} as in the previous proof, one easily verifies that every commutator in $[H,H]$ is a product of commutators of the form $[x_k, z]$ or $[y,z]$ with $y,z \in G'$. In particular, we have $H'\le [G,G'] = G_3$. 

Therefore $H$ has nilpotency class at most $c-1 \le p-1$, and is therefore regular.
Thus Proposition \ref{prop:expge} yields that for any $a\in \langle x_k\rangle$ and $b\in G'$, the order of $ab$ does not exceed the maximum of the orders of $a$ and $b$. However,  $ord(a)$ divides $e$ by definition, and $ord(b)$ divides $e$ by Lemma \ref{lem:main}. In total $(ab)^e=1$, and so $\exp(H)$ divides $e$. Furthermore $H$ is a normal subgroup of $G$ (as $G/G'$ is abelian). Finally, $H$ has a partial complement in $G$, namely $U:=\langle x_1,...,x_{k-1}\rangle$, which is strictly smaller then $G$ by definition of $\{x_1,...,x_k\}$.
Of course $U$ is then of generator exponent dividing $e$, again by definition, and of nilpotency class $\le c$. Inductively, $r(U)$ divides $e$, and since $\exp(H)$ divides $e$ by the above, it follows from Lemma \ref{lem:exp_char} that $r(G)\le e$. Since always $r(G)\ge e$, the assertion follows.
\end{proof}
 \noindent
In particular, we get the following, which due to Proposition \ref{solv_cyclic_decomp}  readily yields Theorem \ref{thm:main}.
\begin{cor}
Let $G$ be a finite nilpotent group of class $c$, and assume that $p\ge c$ where $p$ is the smallest prime divisor of $|G|$. 
Then $r(G) = \textrm{ge}(G)$.
%Then there exist infinitely many Galois extensions $K|\qq$ of degree $[K:\qq]$ at most $e$ such that $G$ occurs as the Galois group of an unramified Galois extension of $K$.
\end{cor}
\begin{proof}
\iffalse
Let $G=P_1\times \cdots \times P_k$, where the $P_j$ are the non-trivial Sylow subgroups of $G$. Obviously, $\textrm{ge}(G) = \prod_{j=1}^k \textrm{ge}(P_j)$. On the other hand, if 
$((N_{ij},P_{ij})\mid i\in \{1,\dots ,s_j\})$ is a series of normal subgroups and partial complements in $P_j$ as in Definition \ref{def:scholzexp}, reaching the minimum $r(P_j) = \textrm{ge}(P_j)$ for $j=1,\dots, k$, then $((\prod_{j=1}^k N_{ij}, \prod_{j=1}^k P_{ij}), \mid i\in \{1,\dots, \max_j s_j\})$ reaches the value $\prod_{j=1}^k \textrm{ge}(P_j) = \textrm{ge}(G)$. Since $r(G)\ge \textrm{ge}(G)$ by definition, equality follows.
\fi
Since a nilpotent group is the direct product of its Sylow subgroups, this follows directly from Theorem \ref{thm:pgroups} together with Lemma \ref{lem:dirprod}.
\end{proof}

\section{Combination with other methods}
\label{sec:last}
The bound $c\le p$ in Theorem \ref{thm:pgroups} is best possible in the sense that there exist $p$-groups of nilpotency class $p+1$ for which $r(G)>\textrm{ge}(G)$, the easiest and smallest example being the dihedral group $D_8$ of order $16$. For other small primes $p$, computer search, e.g.\ with Magma (\cite{Magma}), provides examples of order $|G| = p^{p+2}$, and it should be possible to give explicit examples for all $p$. E.g., for $p=3$,  
%(resp., $p=5$), 
six out of 67 
%(resp., .. out of 34297) 
 groups of order $p^{p+2}$ have nilpotency class $p+1$, and out of those, two fail to satisfy $r(G) = \textrm{ge}(G)$. For such groups, additional ideas are required to answer Questions \ref{conj:e3} and \ref{conj:d}. One thing to note is that, due to the nature of Shafarevich's method, one may improve on the bound $r(G)$ by replacing any value $r(G_i)$ in the iteration process of Definition \ref{def:scholzexp} by any known upper bound for $e(\mathbb{Q}, G_i)$, in case such a bound better than $r(G_i)$ is known. E.g., $e(D_n, \mathbb{Q}) =2$ is known from class field theory (see, e.g., \cite{Yamamoto}). Substituting this value in the definition of $r(G)$ whenever a dihedral $G_i$ occurs (and calling the thus altered constant $r'(G)$ for the moment) yields $e(\mathbb{Q},G) = r'(G) = \textrm{ge}(G)$ for six of the eight nilpotent groups of order $<64$ which fulfill $r(G) > \textrm{ge}(G)$. The two remaining cases ($U_1=$SmallGroup(32,19) and $U_2=$SmallGroup(32,20) in Magma's database) both have generator exponent $4$, and $r'(U_i)=8$. However, they both embed as index-$2$ normal subgroups into $G=$SmallGroup(64,189), which has $r'(G)=2$. So there exist tame $G$-extensions $L/\mathbb{Q}$ with all inertia groups of order $2$. Choose a quadratic extension $K/\mathbb{Q}$, without loss of generality linearly disjoint to the fixed fields $F_i$ of $U_i$ in $L$ ($i=1,2$), such that $LK/K$ is an unramified $G$-extension. Then $LK/F_iK$ is an unramified $U_i$-extension and $F_iK/\mathbb{Q}$ is Galois of group $C_2\times C_2$, whose order equals $\textrm{ge}(U_i)$. We have therefore at least answered Question \ref{conj:d} for $U_i$, and in total have obtained (aided by computer calculation):
 
 \begin{thm}
 Question \ref{conj:d} has a positive answer for all nilpotent groups of order $<64$.
 \end{thm}

\end{document}